\documentclass[a4paper,oneside,10pt]{article}%
\usepackage{amsmath}
\usepackage{amsfonts}
\usepackage{amssymb}
\usepackage{graphicx}
\usepackage{color}
\usepackage[square,numbers,sort&compress]{natbib}%
\providecommand{\U}[1]{\protect \rule{.1in}{.1in}}

\newtheorem{theorem}{Theorem}[section]

\newtheorem{definition}[theorem]{Definition}

\newtheorem{lemma}[theorem]{Lemma}

\newtheorem{proposition}[theorem]{Proposition}
\newtheorem{remark}[theorem]{Remark}

\newenvironment{proof}[1][Proof]{\noindent \textbf{#1.} }{\  $\Box$}
\numberwithin{equation}{section}

\begin{document}

\title{Dynamic programming principle for stochastic optimal control problem under
degenerate $G$-expectation}
\author{Xiaojuan Li\thanks{Zhongtai Securities Institute for Financial Studies,
Shandong University, Jinan 250100, China, lixiaojuan@mail.sdu.edu.cn, Research supported by Natural Science Foundation of Shandong Province (No.
ZR2014AP005). } }
\maketitle

\textbf{Abstract}. In this paper, we study a stochastic optimal control
problem under degenerate $G$-expectation. By using implied partition method,
we show that the approximation result for admissible controls still hold.
Based on this result, we prove that the value function is deterministic, and
obtain the dynamic programming principle. Furthermore, we prove that the value
function is the unique viscosity solution to the related HJB equation under
degenerate case.

{\textbf{Key words}. } $G$-expectation, Dynamic programming principle,
Hamilton-Jacobi-Bellman equation, Stochastic optimal control

\textbf{AMS subject classifications.} 93E20, 60H10, 35K15

\addcontentsline{toc}{section}{\hspace*{1.8em}Abstract}

\section{Introduction}
It is well-known that many economic and financial problems involve volatility uncertainty
(see \cite{EJ-1, EJ-2}), which is characterized by a family of nondominated probability measures.
In this case, this kind of problems cannot be modeled within a probability space framework. So we need a
new framework to deal with it. Motivated by the study of this problem, Peng \cite{P07a, P08a} established the theory of $G$-expectation
$\mathbb{\hat{E}}[\cdot]$. The $G$-Brownian motion $B=(B^{1},\ldots,B^{d})^{T}$ and It\^{o}'s integral with
respect to $B$ were constructed. Moreover, the theory of stochastic
differential equation driven by $G$-Brownian motion ($G$-SDE) has been established.

Stochastic optimal control problems have important applications in economy and finance, such as the utility maximization problems in finance.
The dynamic programming principle (DPP), originated by Bellman in the 1950s, is a powerful tool to solve stochastic optimal control problems.
Under the probability space framework, DPP and the related Hamilton-Jacobi-Bellman (HJB) equation have been intensively
studied by a lot of researchers for various kinds of stochastic optimal control problems (see books \cite{MY, J.Yong} and the references therein).

Under the $G$-expectation framework, Hu and Ji \cite{HJ1} first investigated
the stochastic recursive optimal control problem under non-degenerate $G$, and
obtained the related DPP and HJB equation. For the application of DPP and HJB,
Fouque et al. \cite{FPW} first studied the portfolio selection with ambiguous
correlation and stochastic volatilities (see \cite{HW, Pu} for further
research results). In addition, for the different formulation and method in
studying problems with volatility uncertainty in finance, we refer the readers
to \cite{DM, DK, MPZ, PZ, STZ11} and the references therein.

In this paper, we study the following stochastic control system under
degenerate $G$:%
\begin{equation}
\left \{
\begin{array}
[c]{rl}%
dX_{s}^{t,x,u}= & b(s,X_{s}^{t,x,u},u_{s})ds+h_{ij}(s,X_{s}^{t,x,u}%
,u_{s})d\langle B^{i},B^{j}\rangle_{s}+\sigma(s,X_{s}^{t,x,u},u_{s})dB_{s},\\
X_{t}^{t,x,u}= & x,
\end{array}
\right.  \label{e1-1}%
\end{equation}
where $(t,x)\in \lbrack0,T]\times \mathbb{R}^{n}$, the control domain $U$ is a
given nonempty compact set of $\mathbb{R}^{m}$, and the set of admissible
controls $(u_{s})_{s\in \lbrack t,T]}$ is denoted by $\mathcal{U}%
[t,T]=M_{G}^{2}(t,T;U)$. The value function is defined by
\begin{equation}
V(t,x):=\underset{u\in \mathcal{U}[t,T]}{ess\inf}\mathbb{\hat{E}}_{t}\left[
\Phi(X_{T}^{t,x,u})+\int_{t}^{T}f(s,X_{s}^{t,x,u},u_{s})ds+\int_{t}^{T}%
g_{ij}(s,X_{s}^{t,x,u},u_{s})d\langle B^{i},B^{j}\rangle_{s}\right]  .
\label{e1-2}%
\end{equation}

Under the non-degenerate $G$, Hu, Wang and Zheng \cite{HWZ} proved that
$I_{[c,c^{\prime})}(B_{t}^{i})\in L_{G}^{2}(\Omega_{t})$ for each
$c<c^{\prime}$ and $i\leq d$. Based on this result, Hu and Ji \cite{HJ1}
showed that $\mathcal{U}[t,T]$ contains enough simple and non-trivial
admissible controls. Furthermore, this kind of controls is dense in
$\mathcal{U}[t,T]$ under the norm in $M_{G}^{2}(t,T;\mathbb{R}^{m})$ (see
Lemma 13 in \cite{HJ1}), which is the key point to prove that the value
function $V(\cdot,\cdot)$ satisfies the DPP. But under the degenerate $G$, if
$B^{i}$ is degenerate, then $I_{[c,c^{\prime})}(B_{t}^{i})\not \in L_{G}%
^{2}(\Omega_{t})$ for each $c<c^{\prime}$ and $t>0$ (see Theorem 4.1 in
\cite{LL}), which is completely different from the non-degenerate case.
Therefore, a natural question is whether the above control problem
(\ref{e1-1})-(\ref{e1-2}) under degenerate $G$ is well-posed.

In order to overcome this difficulty, we need to assume that there exists a
non-degenerate $B^{i^{\ast}}$. By using implied partition method which was
proposed in \cite{HJ1} to find optimal control, we obtain that $u\in
\mathcal{U}[t,T]$ can be approximated by a sequence of $u^{k}\in
\mathbb{U}[t,T]$, $k\geq1$, under the norm in $M_{G}^{2}(t,T;\mathbb{R}^{m})$
(see Lemma \ref{le14}). Based on this result, we prove that the value function
$V(\cdot,\cdot)$ is deterministic and satisfies the DPP. Furthermore, we show
that $V(\cdot,\cdot)$ is the unique viscosity solution to the related HJB
equation under degenerate case.

This paper is organized as follows. In Section 2, we recall some basic notions
and results of $G$-expectation. The formulation of our stochastic optimal
control problem under degenerate $G$ is given in Section 3. In Section 4, we
prove that the value function $V(\cdot,\cdot)$ is deterministic, and obtain
the DPP. In Section 5, we show that $V(\cdot,\cdot)$ is the unique viscosity
solution to the related second-order fully nonlinear HJB equation under
degenerate case.

\section{Preliminaries}

We recall some basic notions and results of $G$-expectation. The readers may
refer to \cite{P2019, HJPS1, HJPS} for more details.

Let $T>0$ be fixed and let $\Omega_{T}=C_{0}([0,T];\mathbb{R}^{d})$ be the
space of $\mathbb{R}^{d}$-valued continuous functions on $[0,T]$ with
$\omega_{0}=0$. The canonical process $B_{t}(\omega):=\omega_{t}$, for
$\omega \in \Omega_{T}$ and $t\in \lbrack0,T]$. For each given $t\in \lbrack0,T]$,
set%
\[
Lip(\Omega_{t}):=\{ \varphi(B_{t_{1}},B_{t_{2}}-B_{t_{1}},\ldots,B_{t_{N}%
}-B_{t_{N-1}}):N\geq1,t_{1}<\cdots<t_{N}\leq t,\varphi \in C_{b.Lip}%
(\mathbb{R}^{d\times N})\},
\]
where $C_{b.Lip}(\mathbb{R}^{d\times N})$ denotes the space of bounded
Lipschitz functions on $\mathbb{R}^{d\times N}$. It is easy to verify that%
\[
Lip(\Omega_{t}):=\{ \phi(B_{t_{1}},B_{t_{2}},\ldots,B_{t_{N}}):N\geq
1,t_{1}<\cdots<t_{N}\leq t,\phi \in C_{b.Lip}(\mathbb{R}^{d\times N})\}.
\]

Let $G:\mathbb{S}_{d}\rightarrow \mathbb{R}$ be a given monotonic and sublinear
function, where $\mathbb{S}_{d}$ denotes the set of $d\times d$ symmetric
matrices. Then there exists a bounded and convex set $\Sigma \subset
\mathbb{S}_{d}^{+}$ such that%
\[
G(A)=\frac{1}{2}\sup_{\gamma \in \Sigma}\mathrm{tr}[A\gamma]\text{ for }%
A\in \mathbb{S}_{d},
\]
where $\mathbb{S}_{d}^{+}$ denotes the set of $d\times d$ nonnegative
matrices. If there exists a $\delta>0$ such that $\gamma \geq \delta I_{d}$ for
any $\gamma \in \Sigma$, $G$ is called non-degenerate. Otherwise, $G$ is called
degenerate. In particular, if $d=1$, then%
\[
G(a)=\frac{1}{2}(\bar{\sigma}^{2}a^{+}-\underline{\sigma}^{2}a^{-})\text{ for
}a\in \mathbb{R},
\]
where $\bar{\sigma}^{2}=\sup \Sigma$ and $\underline{\sigma}^{2}=\inf \Sigma
\geq0$. Under this case, $G$ is degenerate iff $\underline{\sigma}^{2}=0$.

Peng \cite{P07a, P08a} constructed the $G$-expectation $\mathbb{\hat{E}%
}:Lip(\Omega_{T})\rightarrow \mathbb{R}$ and the conditional $G$-expectation
$\mathbb{\hat{E}}_{t}:Lip(\Omega_{T})\rightarrow Lip(\Omega_{t})$ as follows:

\begin{description}
\item[(i)] For $s\leq t\leq T$ and $\varphi \in C_{b.Lip}(\mathbb{R}^{d})$,
define $\mathbb{\hat{E}}[\varphi(B_{t}-B_{s})]=u(t-s,0)$, where $u$ is the
viscosity solution (see \cite{CIP}) of the following $G$-heat equation:%
\[
\partial_{t}u-G(\partial_{xx}^{2}u)=0,\ u(0,x)=\varphi(x).
\]

\item[(ii)] For $X=\varphi(B_{t_{1}},B_{t_{2}}-B_{t_{1}},\ldots,B_{t_{N}%
}-B_{t_{N-1}})\in Lip(\Omega_{T})$, define
\[
\mathbb{\hat{E}}_{t_{i}}[X]=\varphi_{i}(B_{t_{1}},\ldots,B_{t_{i}}-B_{t_{i-1}%
})\text{ for }i=N-1,\ldots,1\text{ and }\mathbb{\hat{E}}[X]=\mathbb{\hat{E}%
}[\varphi_{1}(B_{t_{1}})],
\]
where $\varphi_{N-1}(x_{1},\ldots,x_{N-1}):=\mathbb{\hat{E}}[\varphi
(x_{1},\ldots,x_{N-1},B_{t_{N}}-B_{t_{N-1}})]$ and
\[
\varphi_{i}(x_{1},\ldots,x_{i}):=\mathbb{\hat{E}}[\varphi_{i+1}(x_{1}%
,\ldots,x_{i},B_{t_{i+1}}-B_{t_{i}})]\text{ for }i=N-2,\ldots,1.
\]

\end{description}

The space $(\Omega_{T},Lip(\Omega_{T}),\mathbb{\hat{E}},(\mathbb{\hat{E}}%
_{t})_{t\in \lbrack0,T]})$ is a consistent sublinear expectation space, where
$\mathbb{\hat{E}}_{0}=\mathbb{\hat{E}}$. The canonical process $(B_{t}%
)_{t\in \lbrack0,T]}$ is called the $G$-Brownian motion under $\mathbb{\hat{E}%
}$.

For each $t\in \lbrack0,T]$, denote by $L_{G}^{p}(\Omega_{t})$ the completion
of $Lip(\Omega_{t})$ under the norm $||X||_{L_{G}^{p}}:=(\mathbb{\hat{E}%
}[|X|^{p}])^{1/p}$ for $p\geq1$. $\mathbb{\hat{E}}_{t}$ can be continuously
extended to $L_{G}^{1}(\Omega_{T})$ under the norm $||\cdot||_{L_{G}^{1}}$.

\begin{theorem}
(\cite{DHP11, HP09}) There exists a weakly compact set of probability measures
$\mathcal{P}$ on $(\Omega_{T},\mathcal{B}(\Omega_{T}))$ such that%
\[
\mathbb{\hat{E}}[X]=\sup_{P\in \mathcal{P}}E_{P}[X]\text{ for all }X\in
L_{G}^{1}(\Omega_{T}).
\]
$\mathcal{P}$ is called a set that represents $\mathbb{\hat{E}}$.
\end{theorem}

For this $\mathcal{P}$, we define capacity%
\[
c(A):=\sup_{P\in \mathcal{P}}P(A)\text{ for }A\in \mathcal{B}(\Omega_{T}).
\]
A set $A\in \mathcal{B}(\Omega_{T})$ is polar if $c(A)=0$. A property holds
\textquotedblleft quasi-surely" (q.s. for short) if it holds outside a polar
set. In the following, we do not distinguish two random variables $X$ and $Y$
if $X=Y$ q.s.

\begin{definition}
Let $M_{G}^{0}(0,T)$ be the space of simple processes in the following form:
for each $N\in \mathbb{N}$ and $0=t_{0}<\cdots<t_{N}=T$,%
\[
\eta_{t}=\sum_{i=0}^{N-1}\xi_{i}I_{[t_{i},t_{i+1})}(t),
\]
where $\xi_{i}\in Lip(\Omega_{t_{i}})$ for $i=0,1,\ldots,N-1$.
\end{definition}

Denote by $M_{G}^{p}(0,T)$ the completion of $M_{G}^{0}(0,T)$ under the norm
$||\eta||_{M_{G}^{p}}:=\left(  \mathbb{\hat{E}}[\int_{0}^{T}|\eta_{t}%
|^{p}dt]\right)  ^{1/p}$ for $p\geq1$. For each $\eta^{i}\in M_{G}^{2}(0,T)$,
$i=1,\ldots,d$, denote $\eta=(\eta^{1},\ldots,\eta^{d})^{T}\in M_{G}%
^{2}(0,T;\mathbb{R}^{d})$, the $G$-It\^{o} integral $\int_{0}^{T}\eta_{t}%
^{T}dB_{t}$ is well defined.

\section{Formulation of the control problem}

Let $U$ be a given nonempty compact set of $\mathbb{R}^{m}$. For each
$t\in \lbrack0,T]$, we denote by
\[
\mathcal{U}[t,T]:=\{u:u\in M_{G}^{2}(t,T;\mathbb{R}^{m})\text{ with values in
}U\}
\]
the set of admissible controls on $[t,T]$.

In the following, we use Einstein summation convention. For each given
$t\in \lbrack0,T]$, $\xi \in L_{G}^{2}(\Omega_{t};\mathbb{R}^{n})=\{(\xi
_{1},\ldots,\xi_{n})^{T}:\xi_{i}\in L_{G}^{2}(\Omega_{t})$, $i\leq n\}$ and
$u\in \mathcal{U}[t,T]$, we consider the following $G$-SDE:%
\begin{equation}
\left \{
\begin{array}
[c]{rl}%
dX_{s}^{t,\xi,u}= & b(s,X_{s}^{t,\xi,u},u_{s})ds+h_{ij}(s,X_{s}^{t,\xi
,u},u_{s})d\langle B^{i},B^{j}\rangle_{s}+\sigma(s,X_{s}^{t,\xi,u}%
,u_{s})dB_{s},\\
X_{t}^{t,\xi,u}= & \xi,
\end{array}
\right.  \label{e3-1}%
\end{equation}
where $s\in \lbrack t,T]$, $\langle B\rangle=(\langle B^{i},B^{j}%
\rangle)_{i,j=1}^{d}$ is the quadratic variation of $B$. The cost function is
defined by%
\begin{equation}
J(t,\xi,u)=\mathbb{\hat{E}}_{t}\left[  \Phi(X_{T}^{t,\xi,u})+\int_{t}%
^{T}f(s,X_{s}^{t,\xi,u},u_{s})ds+\int_{t}^{T}g_{ij}(s,X_{s}^{t,\xi,u}%
,u_{s})d\langle B^{i},B^{j}\rangle_{s}\right]  . \label{e3-2}%
\end{equation}

Suppose that $b$, $h_{ij}:[0,T]\times \mathbb{R}^{n}\times U\rightarrow
\mathbb{R}^{n}$, $\sigma:[0,T]\times \mathbb{R}^{n}\times U\rightarrow
\mathbb{R}^{n\times d}$, $\Phi:\mathbb{R}^{n}\rightarrow \mathbb{R}$, $f$,
$g_{ij}:[0,T]\times \mathbb{R}^{n}\times U\rightarrow \mathbb{R}$ are
deterministic functions and satisfy the following conditions:

\begin{description}
\item[(H1)] There exists a constant $L>0$ such that for any $(s,x,v)$,
$(s,x^{\prime},v^{\prime})\in \lbrack0,T]\times \mathbb{R}^{n}\times U$,%
\[%
\begin{array}
[c]{l}%
|b(s,x,v)-b(s,x^{\prime},v^{\prime})|+|h_{ij}(s,x,v)-h_{ij}(s,x^{\prime
},v^{\prime})|+|\sigma(s,x,v)-\sigma(s,x^{\prime},v^{\prime})|\\
\leq L(|x-x^{\prime}|+|v-v^{\prime}|),\\
|f(s,x,v)-f(s,x^{\prime},v^{\prime})|+|g_{ij}(s,x,v)-g_{ij}(s,x^{\prime
},v^{\prime})|+|\Phi(x)-\Phi(x^{\prime})|\\
\leq L[(1+|x|+|x^{\prime}|)|x-x^{\prime}|+|v-v^{\prime}|];
\end{array}
\]

\item[(H2)] $h_{ij}=h_{ji}$ and $g_{ij}=g_{ji}$; $b,$ $h_{ij},$ $\sigma,$ $f,$
$g_{ij}$ are continuous in $s$.
\end{description}

We have the following theorems.

\begin{theorem}
(\cite{P2019}) Let Assumptions (H1) and (H2) hold. Then, for each $\xi \in
L_{G}^{2}(\Omega_{t};\mathbb{R}^{n})$ and $u\in \mathcal{U}[t,T]$, there exists
a unique solution $X\in M_{G}^{2}(t,T;\mathbb{R}^{n})$ for the $G$-SDE
(\ref{e3-1}).
\end{theorem}

\begin{theorem}
\label{th-ex}(\cite{HJ1, P2019}) Let Assumptions (H1) and (H2) hold, and let
$\xi,$ $\xi^{\prime}\in L_{G}^{p}(\Omega_{t};\mathbb{R}^{n})$ with $p\geq2$
and $u,$ $v\in \mathcal{U}[t,T]$. Then, for each $\delta \in \lbrack0,T-t]$, we
have%
\[%
\begin{array}
[c]{l}%
\mathbb{\hat{E}}_{t}[|X_{t+\delta}^{t,\xi,u}-X_{t+\delta}^{t,\xi^{\prime}%
,v}|^{2}]\leq C(|\xi-\xi^{\prime}|^{2}+\mathbb{\hat{E}}_{t}[\int_{t}%
^{t+\delta}|u_{s}-v_{s}|^{2}ds]),\\
\mathbb{\hat{E}}_{t}[|X_{t+\delta}^{t,\xi,u}|^{p}]\leq C(1+|\xi|^{p}),\\
\mathbb{\hat{E}}_{t}\left[  \underset{s\in \lbrack t,t+\delta]}{\sup}%
|X_{s}^{t,\xi,u}-\xi|^{p}\right]  \leq C(1+|\xi|^{p})\delta^{p/2},
\end{array}
\]
where $C>0$ depends on $T$, $\bar{\sigma}^{2}=\sup \{|\gamma|:\gamma \in
\Sigma \}$, $p$ and $L$.
\end{theorem}

Our stochastic optimal control problem is to find $u\in \mathcal{U}[t,T]$ which
minimizes the cost function $J(t,x,u)$ for each given $t\in \lbrack0,T]$ and
$x\in \mathbb{R}^{n}$. For this purpose, we need the following definition of
\ essential infimum of $\{J(t,x,u):u\in \mathcal{U}[t,T]\}$.

\begin{definition}
(\cite{HJ1}) \label{def-3-1}The essential infimum of $\{J(t,x,u):u\in
\mathcal{U}[t,T]\}$, denoted by $\underset{u\in \mathcal{U}[t,T]}{ess\inf
}J(t,x,u)$, is a random variable $\zeta \in L_{G}^{2}(\Omega_{t})$ satisfying:

\begin{description}
\item[(i)] for any $u\in \mathcal{U}[t,T]$, $\zeta \leq J(t,x,u)$ q.s.;

\item[(ii)] if $\eta$ is a random variable satisfying $\eta \leq J(t,x,u)$ q.s.
for any $u\in \mathcal{U}[t,T]$, then $\zeta \geq \eta$ q.s.
\end{description}
\end{definition}

For each $(t,x)\in \lbrack0,T]\times \mathbb{R}^{n}$, we define the value
function%
\begin{equation}
V(t,x):=\underset{u\in \mathcal{U}[t,T]}{ess\inf}J(t,x,u). \label{e3-3}%
\end{equation}

In this paper, we consider the degenerate $d$-dimensional $G$-Brownian motion
$B=(B^{1},\ldots,B^{d})^{T}$ with $d\geq2$. For each given $\beta \in
\mathbb{R}^{d}$, by Proposition 3.1.5 in \cite{P2019}, we know that $B^{\beta
}:=\beta^{T}B$ is a $1$-dimensional $G_{\beta}$-Brownian motion with%
\begin{equation}
G_{\beta}(a)=G(\beta \beta^{T})a^{+}+G(-\beta \beta^{T})a^{-}\text{ for }%
a\in \mathbb{R}. \label{e3-4}%
\end{equation}
In particular, $B^{i}$ is a $1$-dimensional $G_{i}$-Brownian motion with%
\[
G_{i}(a)=\frac{1}{2}(\bar{\sigma}_{i}^{2}a^{+}-\underline{\sigma}_{i}^{2}%
a^{-})\text{ for }a\in \mathbb{R},
\]
where $\bar{\sigma}_{i}^{2}=2G(e_{i}e_{i}^{T})$, $\underline{\sigma}_{i}%
^{2}=-2G(-e_{i}e_{i}^{T})$, $\{e_{i}:i\leq d\}$ is the standard basis of
$\mathbb{R}^{d}$. If $\underline{\sigma}_{i}^{2}=0$ for any $i\leq d$, then
$I_{[a,b]}(B_{t}^{i})\not \in L_{G}^{1}(\Omega_{t})$ for any $i\leq d$ and
non-empty interval $[a,b]$ by Theorem 4.1 in \cite{LL}. Under the case
$U=\{0,1\}$, $\mathcal{U}[t,T]$ contains only deterministic controls, which
causes our control problem (\ref{e3-3}) to be ill-posed. Thus we need the
following assumption:

\begin{description}
\item[(H3)] There exists an $i^{\ast}\leq d$ such that $\underline{\sigma
}_{i^{\ast}}^{2}=-2G(-e_{i^{\ast}}e_{i^{\ast}}^{T})=\inf_{\gamma \in \Sigma}\{
\gamma_{i^{\ast}i^{\ast}}\}>0$, where $\gamma=(\gamma_{ij})_{i,j=1}^{d}$.
\end{description}

\begin{remark}
The assumption (H3) implies that $B^{i^{\ast}}$ is a $1$-dimensional
non-degenerate $G_{i^{\ast}}$-Brownian motion. It is clear that
\[
G((a_{ij})_{i,j=1}^{d})=\frac{1}{2}\sum_{i=1}^{d-1}\bar{\sigma}_{i}^{2}%
a_{ii}{}^{+}+\frac{1}{2}(\bar{\sigma}_{d}^{2}a_{dd}^{+}-\underline{\sigma}%
_{d}^{2}a_{dd}^{-})
\]
satisfies (H3), where $\bar{\sigma}_{i}^{2}>0$ for $i\leq d-1$ and
$\bar{\sigma}_{d}^{2}\geq \underline{\sigma}_{d}^{2}>0$, and $B^{d}$ is a
$1$-dimensional non-degenerate $G_{d}$-Brownian motion.
\end{remark}

In the following we will prove that $V(\cdot,\cdot)$ is deterministic.
Furthermore, we will obtain the dynamic programming principle and the related
fully nonlinear HJB equation under degenerate case.

\section{Dynamic programming principle}

We use the following notations: for each given $0\leq t\leq s\leq T$,%

\[%
\begin{array}
[c]{l}%
Lip(\Omega_{s}^{t}):=\{ \varphi(B_{t_{1}}-B_{t},\ldots,B_{t_{N}}-B_{t}%
):N\geq1,t_{1},\ldots,t_{N}\in \lbrack t,s],\varphi \in C_{b.Lip}(\mathbb{R}%
^{d\times N})\};\\
L_{G}^{p}(\Omega_{s}^{t}):=\{ \text{the completion of }Lip(\Omega_{s}%
^{t})\text{ under the norm }||\cdot||_{L_{G}^{p}}\} \text{, }p\geq1;\\
M_{G}^{0,t}(t,T):=\{ \eta_{s}=\sum_{k=0}^{N-1}\xi_{k}I_{[t_{k},t_{k+1}%
)}(s):t=t_{0}<\cdots<t_{N}=T,\xi_{k}\in Lip(\Omega_{t_{k}}^{t})\};\\
M_{G}^{p,t}(t,T):=\{ \text{the completion of }M_{G}^{0,t}(t,T)\text{ under the
norm }||\cdot||_{M_{G}^{p}}\} \text{, }p\geq1;\\
\mathcal{U}^{t}[t,T]:=\{u:u\in M_{G}^{2,t}(t,T;\mathbb{R}^{m})\text{ with
values in }U\};\\
\mathbb{U}[t,T]:=\{u=\sum_{k=1}^{N}I_{A_{k}}u^{k}:N\geq1,u^{k}\in
\mathcal{U}^{t}[t,T],I_{A_{k}}\in L_{G}^{2}(\Omega_{t}),(A_{k})_{k=1}%
^{N}\text{ is a partition of }\Omega \}.
\end{array}
\]

For simplicity, the constant $C$ will change from line to line in the
following. In order to prove that $V(\cdot,\cdot)$ is deterministic, we need
the following lemmas.

\begin{lemma}
\label{le11}Let Assumption (H3) hold. Then there exists a constant $\lambda>0$
such that $B^{i}+\lambda B^{i^{\ast}}$ is non-degenerate for $i\leq d$ and
$i\not =i^{\ast}$.
\end{lemma}

\begin{proof}
By (\ref{e3-4}), $B^{i}+\lambda B^{i^{\ast}}$ is non-degenerate if and only if%
\[
-2G(-(e_{i}+\lambda e_{i^{\ast}})(e_{i}+\lambda e_{i^{\ast}})^{T}%
)=\inf_{\gamma \in \Sigma}\{ \gamma_{ii}+2\lambda \gamma_{ii^{\ast}}+\lambda
^{2}\gamma_{i^{\ast}i^{\ast}}\}>0.
\]
Since $\Sigma$ is bounded, we know $\alpha:=\sup_{\gamma \in \Sigma}%
|\gamma|<\infty$. Taking $\lambda=(2\alpha+1)(\underline{\sigma}_{i^{\ast}%
}^{2})^{-1}$ and noting that $\gamma_{ii}\geq0$, we obtain%
\[
\inf_{\gamma \in \Sigma}\{ \gamma_{ii}+2\lambda \gamma_{ii^{\ast}}+\lambda
^{2}\gamma_{i^{\ast}i^{\ast}}\} \geq \lambda^{2}\underline{\sigma}_{i^{\ast}%
}^{2}-2\lambda \alpha=\lambda>0.
\]
Thus $B^{i}+\lambda B^{i^{\ast}}$ is non-degenerate for each $i\not =i^{\ast}$.
\end{proof}

\begin{lemma}
\label{le12}Let Assumption (H3) hold and let $\xi \in L_{G}^{2}(\Omega_{s})$
with fixed $s\in \lbrack t,T]$. Then there exists a sequence $\xi^{k}%
=\sum_{j=1}^{N_{k}}\sum_{l=1}^{\bar{N}_{k}}x_{jl}^{k}I_{A_{j}^{k}}I_{\bar
{A}_{l}^{k}}$, $k\geq1$, such that%
\[
\lim_{k\rightarrow \infty}\mathbb{\hat{E}}\left[  |\xi-\xi^{k}|^{2}\right]
=0,
\]
where $x_{jl}^{k}\in \mathbb{R}$, $I_{A_{j}^{k}}\in L_{G}^{2}(\Omega_{t})$,
$I_{\bar{A}_{l}^{k}}\in L_{G}^{2}(\Omega_{s}^{t})$, $j\leq N_{k}$, $l\leq
\bar{N}_{k}$, $k\geq1$, $(A_{j}^{k})_{i=1}^{N_{k}}$ is a $\mathcal{B}%
(\Omega_{t})$-partition of $\Omega$, and $(\bar{A}_{l}^{k})_{l=1}^{\bar{N}%
_{k}}$ is a $\mathcal{B}(\Omega_{s}^{t})$-partition of $\Omega$.
\end{lemma}

\begin{proof}
Since $L_{G}^{2}(\Omega_{s})$ is the completion of $Lip(\Omega_{s})$ under the
norm $||\cdot||_{L_{G}^{2}}$, we only need to prove the case
\begin{equation}
\xi=\varphi(B_{t_{1}},B_{t_{2}}-B_{t_{1}},\ldots,B_{t_{N}}-B_{t_{N-1}}),
\label{e3-5}%
\end{equation}
where $N\geq1$, $0<t_{1}<\cdots<t_{N}\leq s$, $t_{i}=t$ for some $i\leq N$,
$\varphi \in C_{b.Lip}(\mathbb{R}^{d\times N})$. In the following, we only
prove the case
\[
\xi=\varphi(B_{t},B_{s}-B_{t})
\]
for simplicity. The proof for (\ref{e3-5}) is similar.

By Lemma \ref{le11}, there exists a constant $\lambda>0$ such that
$B^{i}+\lambda B^{i^{\ast}}$ is non-degenerate for each $i\not =i^{\ast}$. Set%
\[
B_{r}^{\lambda}=(B_{r}^{1}+\lambda B_{r}^{i^{\ast}},\ldots,B_{r}^{i^{\ast}%
-1}+\lambda B_{r}^{i^{\ast}},B_{r}^{i^{\ast}},B_{r}^{i^{\ast}+1}+\lambda
B_{r}^{i^{\ast}},\ldots,B_{r}^{d}+\lambda B_{r}^{i^{\ast}})^{T}.
\]
It follows from Theorem 3.20 in \cite{HWZ} that%
\begin{equation}
I_{\{B_{t}^{\lambda}\in \lbrack c,c^{\prime})\}}\in L_{G}^{2}(\Omega_{t})\text{
and }I_{\{B_{s}^{\lambda}-B_{t}^{\lambda}\in \lbrack c,c^{\prime})\}}\in
L_{G}^{2}(\Omega_{s}^{t}) \label{e3-6}%
\end{equation}
for any $c=(c_{1},\ldots,c_{d})^{T}$, $\bar{c}=(\bar{c}_{1},\ldots,\bar{c}%
_{d})^{T}\in \mathbb{R}^{d}$ with $c<c^{\prime}$. For each $k\geq1$, it is easy
to find a finite number of disjoint intervals $[c^{j,k},\bar{c}^{j,k})$,
$j=1$,$\ldots$,$N_{k}-1$, such that $|\bar{c}^{j,k}-c^{j,k}|<k^{-1}$ and
$[-ke,ke)=\cup_{j\leq N_{k}-1}[c^{j,k},\bar{c}^{j,k})$ with $e=[1,\ldots
,1]^{T}\in \mathbb{R}^{d}$. Define%
\[
A_{j}^{k}=\{B_{t}^{\lambda}\in \lbrack c^{j,k},\bar{c}^{j,k})\} \text{, }%
\bar{A}_{j}^{k}=\{B_{s}^{\lambda}-B_{t}^{\lambda}\in \lbrack c^{j,k},\bar
{c}^{j,k})\} \text{ for }j\leq N_{k}-1\text{,}%
\]
and $A_{N_{k}}^{k}=\Omega \backslash \cup_{j\leq N_{k}-1}A_{j}^{k}$, $\bar
{A}_{N_{k}}^{k}=\Omega \backslash \cup_{j\leq N_{k}-1}\bar{A}_{j}^{k}$. It is
easy to verify that $(A_{j}^{k})_{i=1}^{N_{k}}$ is a $\mathcal{B}(\Omega_{t}%
)$-partition of $\Omega$, and $(\bar{A}_{j}^{k})_{j=1}^{N_{k}}$ is a
$\mathcal{B}(\Omega_{s}^{t})$-partition of $\Omega$. By (\ref{e3-5}), we know
that $I_{A_{j}^{k}}\in L_{G}^{2}(\Omega_{t})$ and $I_{\bar{A}_{j}^{k}}\in
L_{G}^{2}(\Omega_{s}^{t})$ for $j\leq N_{k}$. Set%
\begin{equation}
\xi^{k}=\sum_{j=1}^{N_{k}}\sum_{l=1}^{N_{k}}\varphi(\tilde{c}^{j,k},\tilde
{c}^{l,k})I_{A_{j}^{k}}I_{\bar{A}_{l}^{k}}, \label{e3-7}%
\end{equation}
where $\tilde{c}^{j,k}=(c_{1}^{j,k}-\lambda c_{i^{\ast}}^{j,k},\ldots
,c_{i^{\ast}-1}^{j,k}-\lambda c_{i^{\ast}}^{j,k},c_{i^{\ast}}^{j,k}%
,c_{i^{\ast}+1}^{j,k}-\lambda c_{i^{\ast}}^{j,k},\ldots,c_{d}^{j,k}-\lambda
c_{i^{\ast}}^{j,k})^{T}$ for $j\leq N_{k}-1$ and $\tilde{c}^{N_{k},k}=0$. For
$j$, $l\leq N_{k}-1$, one can check that
\begin{align*}
|\xi-\xi^{k}|I_{A_{j}^{k}}I_{\bar{A}_{l}^{k}}  &  \leq L_{\varphi}%
(|B_{t}-\tilde{c}^{j,k}|+|B_{s}-B_{t}-\tilde{c}^{l,k}|)I_{A_{j}^{k}}I_{\bar
{A}_{l}^{k}}\\
&  \leq CL_{\varphi}(|B_{t}^{\lambda}-c^{j,k}|+|B_{s}^{\lambda}-B_{t}%
^{\lambda}-c^{l,k}|)I_{A_{j}^{k}}I_{\bar{A}_{l}^{k}},
\end{align*}
where $L_{\varphi}$ is the Lipschitz constant of $\varphi$, the constant $C>0$
depends on $\lambda$ and $d$. Thus, we obtain%
\begin{align*}
|\xi-\xi^{k}|  &  \leq CL_{\varphi}\frac{2}{k}+2M_{\varphi}(I_{A_{N_{k}}^{k}%
}+I_{\bar{A}_{N_{k}}^{k}})\\
&  \leq CL_{\varphi}\frac{2}{k}+\frac{2M_{\varphi}}{k}(|B_{t}^{\lambda
}|+|B_{s}^{\lambda}-B_{t}^{\lambda}|),
\end{align*}
where $M_{\varphi}$ is the bound of $\varphi$. From this, we can get
$\mathbb{\hat{E}}\left[  |\xi-\xi^{k}|^{2}\right]  \leq Ck^{-2}$, which
implies the desired result.
\end{proof}

\begin{remark}
In the above proof, the partition of $B^{i}$ is deduced by the partition of
$B^{i}+\lambda B^{i^{\ast}}$ and $B^{i^{\ast}}$. So, the random variable
$\xi^{k}$ defined in (\ref{e3-7}) is called an implied partition of $\xi$.
\end{remark}

\begin{lemma}
\label{le14}Let Assumption (H3) hold and let $u\in \mathcal{U}[t,T]$ be given.
Then there exists a sequence $(u^{k})_{k\geq1}$ in $\mathbb{U}[t,T]$ such that%
\[
\lim_{k\rightarrow \infty}\mathbb{\hat{E}}\left[  \int_{t}^{T}|u_{s}-u_{s}%
^{k}|^{2}ds\right]  =0.
\]

\end{lemma}

\begin{proof}
Since $u\in M_{G}^{2}(t,T;\mathbb{R}^{m})$, there exists a sequence
\[
\bar{u}_{s}^{k}=\sum_{i=0}^{N-1}\xi_{i}^{k}I_{[t_{i}^{k},t_{i+1}^{k}%
)}(s)\text{, }t=t_{0}^{k}<\cdots<t_{N}^{k}=T\text{, }\xi_{i}^{k}\in
Lip(\Omega_{t_{i}^{k}};\mathbb{R}^{m})\text{, }k\geq1\text{,}%
\]
such that $\mathbb{\hat{E}}\left[  \int_{t}^{T}|u_{s}-\bar{u}_{s}^{k}%
|^{2}ds\right]  \rightarrow0$ as $k\rightarrow \infty$. Note that $I_{AB}%
=I_{A}I_{B}\in L_{G}^{2}(\Omega_{t})$ if $I_{A}$, $I_{B}\in L_{G}^{2}%
(\Omega_{t})$, then, by Lemma \ref{le12}, we can find%
\[
\bar{\xi}_{i}^{k}=\sum_{j=1}^{N_{k}}\sum_{l=1}^{\bar{N}_{i,k}}x_{jl}%
^{i,k}I_{A_{j}^{k}}I_{\bar{A}_{l}^{i,k}}\text{ for }i=0,\ldots,N-1,
\]
such that $x_{jl}^{i,k}\in \mathbb{R}^{m}$, $I_{A_{j}^{k}}\in L_{G}^{2}%
(\Omega_{t})$, $I_{\bar{A}_{l}^{i,k}}\in L_{G}^{2}(\Omega_{t_{i}}^{t})$,
$j\leq N_{k}$, $l\leq \bar{N}_{i,k}$, $i\leq N-1$, $(A_{j}^{k})_{i=1}^{N_{k}}$
is a $\mathcal{B}(\Omega_{t})$-partition of $\Omega$, $(\bar{A}_{l}%
^{i,k})_{l=1}^{\bar{N}_{i,k}}$ is a $\mathcal{B}(\Omega_{t_{i}}^{t}%
)$-partition of $\Omega$ and%
\[
\mathbb{\hat{E}}\left[  |\xi_{i}^{k}-\bar{\xi}_{i}^{k}|^{2}\right]  <\frac
{1}{k}\text{ for }i=0,\ldots,N-1.
\]
Set $\tilde{u}_{s}^{k}=\sum_{i=0}^{N-1}\bar{\xi}_{i}^{k}I_{[t_{i}^{k}%
,t_{i+1}^{k})}(s)$. Then we have%
\begin{align*}
\mathbb{\hat{E}}\left[  \int_{t}^{T}|u_{s}-\tilde{u}_{s}^{k}|^{2}ds\right]
&  \leq2\mathbb{\hat{E}}\left[  \int_{t}^{T}|u_{s}-\bar{u}_{s}^{k}%
|^{2}ds\right]  +2\left[  \int_{t}^{T}|\bar{u}_{s}^{k}-\tilde{u}_{s}^{k}%
|^{2}ds\right] \\
&  \leq2\mathbb{\hat{E}}\left[  \int_{t}^{T}|u_{s}-\bar{u}_{s}^{k}%
|^{2}ds\right]  +2\sum_{i=0}^{N-1}\mathbb{\hat{E}}\left[  |\xi_{i}^{k}%
-\bar{\xi}_{i}^{k}|^{2}\right]  (t_{i+1}^{k}-t_{i}^{k})\\
&  \leq2\mathbb{\hat{E}}\left[  \int_{t}^{T}|u_{s}-\bar{u}_{s}^{k}%
|^{2}ds\right]  +\frac{2(T-t)}{k}\rightarrow0
\end{align*}
as $k\rightarrow \infty$. Since $U$ is a nonempty compact set of $\mathbb{R}%
^{m}$, for each $x_{jl}^{i,k}$, there exists a $\tilde{x}_{jl}^{i,k}$ such
that%
\[
|x_{jl}^{i,k}-\tilde{x}_{jl}^{i,k}|=\inf \{|x_{jl}^{i,k}-x|:x\in U\}.
\]
Set%
\[
\tilde{\xi}_{i}^{k}=\sum_{j=1}^{N_{k}}\sum_{l=1}^{\bar{N}_{i,k}}\tilde{x}%
_{jl}^{i,k}I_{A_{j}^{k}}I_{\bar{A}_{l}^{i,k}}\text{ for }i=0,\ldots,N-1,
\]
and%
\[
u_{s}^{k}=\sum_{i=0}^{N-1}\tilde{\xi}_{i}^{k}I_{[t_{i}^{k},t_{i+1}^{k})}(s).
\]
It is easy to check that $u^{k}\in \mathbb{U}[t,T]$. Since $u_{s}\in U$, we
know $|\tilde{u}_{s}^{k}-u_{s}^{k}|\leq|\tilde{u}_{s}^{k}-u_{s}|$. Thus%
\begin{align*}
\mathbb{\hat{E}}\left[  \int_{t}^{T}|u_{s}-u_{s}^{k}|^{2}ds\right]   &
\leq2\mathbb{\hat{E}}\left[  \int_{t}^{T}|u_{s}-\tilde{u}_{s}^{k}%
|^{2}ds\right]  +2\mathbb{\hat{E}}\left[  \int_{t}^{T}|\tilde{u}_{s}^{k}%
-u_{s}^{k}|^{2}ds\right] \\
&  \leq4\mathbb{\hat{E}}\left[  \int_{t}^{T}|u_{s}-\tilde{u}_{s}^{k}%
|^{2}ds\right]  \rightarrow0
\end{align*}
as $k\rightarrow \infty$, which implies the desired result.
\end{proof}

\begin{theorem}
\label{th15}Let Assumptions (H1)-(H3) hold. Then the value function $V(t,x)$
exists for each $(t,x)\in \lbrack0,T]\times \mathbb{R}^{n}$ and%
\begin{equation}
V(t,x):=\underset{v\in \mathcal{U}^{t}[t,T]}{\inf}J(t,x,v). \label{e3-8}%
\end{equation}

\end{theorem}

\begin{proof}
For each $v\in \mathcal{U}^{t}[t,T]$, it is easy to deduce $X_{s}^{t,x,v}\in
L_{G}^{p}(\Omega_{s}^{t})$ for any $p\geq2$, which implies that $J(t,x,v)$ is
a constant. Since $\mathcal{U}^{t}[t,T]\subset \mathcal{U}[t,T]$, by Definition
\ref{def-3-1}, we only need to prove that, for any fixed $u\in \mathcal{U}%
[t,T]$,%
\begin{equation}
J(t,x,u)\geq \underset{v\in \mathcal{U}^{t}[t,T]}{\inf}J(t,x,v),\text{ q.s.}
\label{e3-9}%
\end{equation}
By Lemma \ref{le14}, there exists a sequence $u_{s}^{k}=\sum_{j=1}^{N_{k}%
}I_{A_{j}^{k}}v_{s}^{j,k}$, $k\geq1$, such that $I_{A_{j}^{k}}\in L_{G}%
^{2}(\Omega_{t})$, $v^{j,k}\in \mathcal{U}^{t}[t,T]$, $(A_{j}^{k})_{j=1}%
^{N_{k}}$ is a partition of $\Omega$ and
\begin{equation}
\mathbb{\hat{E}}\left[  \int_{t}^{T}|u_{s}^{k}-u_{s}|^{2}ds\right]
\rightarrow0\text{ as }k\rightarrow \infty \text{.} \label{e3-10}%
\end{equation}
It is easy to check that $X_{s}^{t,x,u^{k}}=\sum_{j=1}^{N_{k}}I_{A_{j}^{k}%
}X_{s}^{t,x,v^{j,k}}$ for $s\in \lbrack t,T]$. Thus%
\begin{align*}
J(t,x,u^{k})  &  =\mathbb{\hat{E}}_{t}\left[  \Phi(X_{T}^{t,x,u^{k}})+\int
_{t}^{T}f(s,X_{s}^{t,x,u^{k}},u_{s}^{k})ds+\int_{t}^{T}g_{ij}(s,X_{s}%
^{t,x,u^{k}},u_{s}^{k})d\langle B^{i},B^{j}\rangle_{s}\right] \\
&  =\sum_{j=1}^{N_{k}}I_{A_{j}^{k}}\mathbb{\hat{E}}_{t}\left[  \Phi
(X_{T}^{t,x,v^{j,k}})+\int_{t}^{T}f(s,X_{s}^{t,x,v^{j,k}},v_{s}^{j,k}%
)ds+\int_{t}^{T}g_{ij}(s,X_{s}^{t,x,v^{j,k}},v_{s}^{j,k})d\langle B^{i}%
,B^{j}\rangle_{s}\right] \\
&  =\sum_{j=1}^{N_{k}}I_{A_{j}^{k}}J(t,x,v^{j,k})\geq \underset{v\in
\mathcal{U}^{t}[t,T]}{\inf}J(t,x,v).
\end{align*}
It follows from (H1) and H\"{o}lder's inequality that%
\begin{align*}
&  \mathbb{\hat{E}}\left[  |J(t,x,u^{k})-J(t,x,u)|\right] \\
&  \leq C\left \{  \left(  1+\left(  \sup_{s\in \lbrack t,T]}\mathbb{\hat{E}%
}\left[  |X_{s}^{t,x,u^{k}}|^{2}+|X_{s}^{t,x,u}|^{2}\right]  \right)
^{1/2}\right)  \left(  \sup_{s\in \lbrack t,T]}\mathbb{\hat{E}}\left[
|X_{s}^{t,x,u^{k}}-X_{s}^{t,x,u}|^{2}\right]  \right)  ^{1/2}\right. \\
&  \left.  \  \  \  \ +\left(  \mathbb{\hat{E}}\left[  \int_{t}^{T}|u_{s}%
^{k}-u_{s}|^{2}ds\right]  \right)  ^{1/2}\right \}  ,
\end{align*}
where $C>0$ depends on $T$, $\bar{\sigma}^{2}$ and $L$. By Theorem \ref{th-ex}
and the above inequality, we obtain%
\begin{equation}
\mathbb{\hat{E}}\left[  |J(t,x,u^{k})-J(t,x,u)|\right]  \leq C(1+|x|)\left(
\mathbb{\hat{E}}\left[  \int_{t}^{T}|u_{s}^{k}-u_{s}|^{2}ds\right]  \right)
^{1/2}, \label{e3-11}%
\end{equation}
where $C>0$ depends on $T$, $\bar{\sigma}^{2}$ and $L$. Combining
(\ref{e3-10}) and (\ref{e3-11}), we get
\begin{equation}
\mathbb{\hat{E}}\left[  |J(t,x,u^{k})-J(t,x,u)|\right]  \rightarrow0\text{ as
}k\rightarrow \infty. \label{e3-12}%
\end{equation}
Since $J(t,x,u^{k})\geq \inf_{v\in \mathcal{U}^{t}[t,T]}J(t,x,v)$, we have%
\begin{equation}
\mathbb{\hat{E}}\left[  \left(  J(t,x,u^{k})-\inf_{v\in \mathcal{U}^{t}%
[t,T]}J(t,x,v)\right)  ^{-}\right]  =0. \label{e3-13}%
\end{equation}
By (\ref{e3-12}) and (\ref{e3-13}), we obtain%
\[
\mathbb{\hat{E}}\left[  \left(  J(t,x,u)-\inf_{v\in \mathcal{U}^{t}%
[t,T]}J(t,x,v)\right)  ^{-}\right]  =0,
\]
which implies $J(t,x,u)\geq \inf_{v\in \mathcal{U}^{t}[t,T]}J(t,x,v)$, q.s. Thus
we obtain (\ref{e3-8}).
\end{proof}

Now we use (\ref{e3-8}) to study the properties of $V(\cdot,\cdot)$ in $x$.

\begin{proposition}
\label{pr16}Let Assumptions (H1)-(H3) hold. Then there exists a constant $C>0$
depending on $T$, $\bar{\sigma}^{2}$ and $L$ such that%
\[
|V(t,x)-V(t,x^{\prime})|\leq C(1+|x|+|x^{\prime}|)|x-x^{\prime}|\text{ and
}|V(t,x)|\leq C(1+|x|^{2})
\]
for $t\in \lbrack0,T]$, $x$, $x^{\prime}\in \mathbb{R}^{n}$.
\end{proposition}

\begin{proof}
For each given $v\in \mathcal{U}^{t}[t,T]$, similar to the proof of
(\ref{e3-11}), we can get%
\begin{align*}
&  |J(t,x,v)-J(t,x^{\prime},v)|\\
&  \leq C\left(  1+\left(  \sup_{s\in \lbrack t,T]}\mathbb{\hat{E}}\left[
|X_{s}^{t,x,v}|^{2}+|X_{s}^{t,x^{\prime},v}|^{2}\right]  \right)
^{1/2}\right)  \left(  \sup_{s\in \lbrack t,T]}\mathbb{\hat{E}}\left[
|X_{s}^{t,x,v}-X_{s}^{t,x^{\prime},v}|^{2}\right]  \right)  ^{1/2}.
\end{align*}
By Theorem \ref{th-ex} and the above inequality, we have%
\[
|J(t,x,v)-J(t,x^{\prime},v)|\leq C(1+|x|+|x^{\prime}|)|x-x^{\prime}|,
\]
where $C>0$ depends on $T$, $\bar{\sigma}^{2}$ and $L$. Thus, by (\ref{e3-8}),
we obtain%
\[
|V(t,x)-V(t,x^{\prime})|\leq \underset{v\in \mathcal{U}^{t}[t,T]}{\sup
}|J(t,x,v)-J(t,x^{\prime},v)|\leq C(1+|x|+|x^{\prime}|)|x-x^{\prime}|.
\]
Note that $U$ is compact, then we can deduce%
\[
|J(t,x,v)|\leq C\left(  1+\sup_{s\in \lbrack t,T]}\mathbb{\hat{E}}\left[
|X_{s}^{t,x,v}|^{2}\right]  \right)  ,
\]
where $C>0$ depends on $T$, $\bar{\sigma}^{2}$ and $L$. By Theorem
\ref{th-ex}, we obtain $|V(t,x)|\leq C(1+|x|^{2})$.
\end{proof}

The following theorem is the dynamic programming principle for control problem
(\ref{e3-3}).

\begin{theorem}
\label{th17}Let Assumptions (H1)-(H3) hold. Then, for each $t<T$, $\delta \leq
T-t$, $x\in \mathbb{R}^{n}$, we have%
\begin{align*}
V(t,x)  &  =\underset{u\in \mathcal{U}[t,t+\delta]}{ess\inf}\mathbb{\hat{E}%
}_{t}\left[  V(t+\delta,X_{t+\delta}^{t,x,u})+\int_{t}^{t+\delta}%
f(s,X_{s}^{t,x,u},u_{s})ds+\int_{t}^{t+\delta}g_{ij}(s,X_{s}^{t,x,u}%
,u_{s})d\langle B^{i},B^{j}\rangle_{s}\right] \\
&  =\underset{v\in \mathcal{U}^{t}[t,t+\delta]}{\inf}\mathbb{\hat{E}}\left[
V(t+\delta,X_{t+\delta}^{t,x,v})+\int_{t}^{t+\delta}f(s,X_{s}^{t,x,v}%
,v_{s})ds+\int_{t}^{t+\delta}g_{ij}(s,X_{s}^{t,x,v},v_{s})d\langle B^{i}%
,B^{j}\rangle_{s}\right]  .
\end{align*}

\end{theorem}

\begin{proof}
Similar to the proof of Theorem \ref{th15}, we have%
\begin{align*}
&  \underset{u\in \mathcal{U}[t,t+\delta]}{ess\inf}\mathbb{\hat{E}}_{t}\left[
V(t+\delta,X_{t+\delta}^{t,x,u})+\int_{t}^{t+\delta}f(s,X_{s}^{t,x,u}%
,u_{s})ds+\int_{t}^{t+\delta}g_{ij}(s,X_{s}^{t,x,u},u_{s})d\langle B^{i}%
,B^{j}\rangle_{s}\right] \\
&  =\underset{v\in \mathcal{U}^{t}[t,t+\delta]}{\inf}\mathbb{\hat{E}}\left[
V(t+\delta,X_{t+\delta}^{t,x,v})+\int_{t}^{t+\delta}f(s,X_{s}^{t,x,v}%
,v_{s})ds+\int_{t}^{t+\delta}g_{ij}(s,X_{s}^{t,x,v},v_{s})d\langle B^{i}%
,B^{j}\rangle_{s}\right]  .
\end{align*}
For each fixed $v\in \mathcal{U}^{t}[t,T]$, we assert that%
\begin{equation}
V(t+\delta,X_{t+\delta}^{t,x,v})\leq \mathbb{\hat{E}}_{t+\delta}\left[
\Phi(X_{T}^{t,x,v})+\int_{t+\delta}^{T}f(s,X_{s}^{t,x,v},v_{s})ds+\int
_{t+\delta}^{T}g_{ij}(s,X_{s}^{t,x,v},v_{s})d\langle B^{i},B^{j}\rangle
_{s}\right]  . \label{e3-14}%
\end{equation}
Since%
\[%
\begin{array}
[c]{rl}%
X_{s}^{t,x,v}= & X_{t+\delta}^{t,x,v}+\int_{t+\delta}^{s}b(r,X_{r}%
^{t,x,v},v_{r})dr+\int_{t+\delta}^{s}h_{ij}(r,X_{r}^{t,x,v},v_{r})d\langle
B^{i},B^{j}\rangle_{r}\\
& +\int_{t+\delta}^{s}\sigma(r,X_{r}^{t,x,v},v_{r})dB_{r},\text{ }s\in \lbrack
t+\delta,T],
\end{array}
\]
we have $X_{s}^{t,x,v}=X_{s}^{t+\delta,X_{t+\delta}^{t,x,v},v}$ for
$s\in \lbrack t+\delta,T]$. Thus inequality (\ref{e3-14}) is equivalent to%
\[
V(t+\delta,X_{t+\delta}^{t,x,v})\leq J(t+\delta,X_{t+\delta}^{t,x,v},v).
\]
By Lemma \ref{le12}, there exists a sequence $\xi^{k}=\sum_{j=1}^{N_{k}}%
x_{j}^{k}I_{A_{j}^{k}}$, $k\geq1$, such that%
\begin{equation}
\lim_{k\rightarrow \infty}\mathbb{\hat{E}}\left[  |X_{t+\delta}^{t,x,v}-\xi
^{k}|^{2}\right]  =0, \label{e3-15}%
\end{equation}
where $x_{j}^{k}\in \mathbb{R}^{n}$, $I_{A_{j}^{k}}\in L_{G}^{2}(\Omega
_{t+\delta})$, $j\leq N_{k}$, $(A_{j}^{k})_{i=1}^{N_{k}}$ is a $\mathcal{B}%
(\Omega_{t+\delta})$-partition of $\Omega$. It is easy to verify that%
\[
X_{s}^{t+\delta,\xi^{k},v}=\sum_{j=1}^{N_{k}}X_{s}^{t+\delta,x_{j}^{k}%
,v}I_{A_{j}^{k}}\text{ for }s\in \lbrack t+\delta,T].
\]
Thus, by the definition of $V(t+\delta,x_{j}^{k})$, we obtain%
\begin{align*}
&  J(t+\delta,\xi^{k},v)\\
&  =\mathbb{\hat{E}}_{t+\delta}\left[  \Phi(X_{T}^{t+\delta,\xi^{k},v}%
)+\int_{t+\delta}^{T}f(s,X_{s}^{t+\delta,\xi^{k},v},v_{s})ds+\int_{t+\delta
}^{T}g_{ij}(s,X_{s}^{t+\delta,\xi^{k},v},v_{s})d\langle B^{i},B^{j}\rangle
_{s}\right] \\
&  =\sum_{j=1}^{N_{k}}I_{A_{j}^{k}}\mathbb{\hat{E}}_{t+\delta}\left[
\Phi(X_{T}^{t+\delta,x_{j}^{k},v})+\int_{t+\delta}^{T}f(s,X_{s}^{t+\delta
,x_{j}^{k},v},v_{s})ds+\int_{t+\delta}^{T}g_{ij}(s,X_{s}^{t+\delta,x_{j}%
^{k},v},v_{s})d\langle B^{i},B^{j}\rangle_{s}\right] \\
&  \geq \sum_{j=1}^{N_{k}}I_{A_{j}^{k}}V(t+\delta,x_{j}^{k}),
\end{align*}
which implies%
\begin{equation}
J(t+\delta,\xi^{k},v)\geq V(t+\delta,\xi^{k}). \label{e3-16}%
\end{equation}
Similar to the proof of (\ref{e3-11}), we can get%
\begin{equation}
\mathbb{\hat{E}}\left[  |J(t+\delta,\xi^{k},v)-J(t+\delta,X_{t+\delta}%
^{t,x,v},v)|\right]  \leq C\left(  1+(\mathbb{\hat{E}}\left[  |\xi^{k}%
|^{2}+|X_{t+\delta}^{t,x,v}|^{2}\right]  )^{1/2}\right)  (\mathbb{\hat{E}%
}\left[  |X_{t+\delta}^{t,x,v}-\xi^{k}|^{2}\right]  )^{1/2}, \label{e3-17}%
\end{equation}
where $C>0$ depends on $T$, $\bar{\sigma}^{2}$ and $L$. By Proposition
\ref{pr16}, we have%
\begin{equation}
\mathbb{\hat{E}}\left[  |V(t+\delta,\xi^{k})-V(t+\delta,X_{t+\delta}%
^{t,x,v})|\right]  \leq C\left(  1+(\mathbb{\hat{E}}\left[  |\xi^{k}%
|^{2}+|X_{t+\delta}^{t,x,v}|^{2}\right]  )^{1/2}\right)  (\mathbb{\hat{E}%
}\left[  |X_{t+\delta}^{t,x,v}-\xi^{k}|^{2}\right]  )^{1/2}, \label{e3-18}%
\end{equation}
where $C>0$ depends on $T$, $\bar{\sigma}^{2}$ and $L$. Thus, by
(\ref{e3-15}), (\ref{e3-17}) and (\ref{e3-18}), we obtain (\ref{e3-14}) by
taking $k\rightarrow \infty$ in (\ref{e3-16}). From (\ref{e3-14}), we can
easily deduce%
\[
J(t,x,v)\geq \mathbb{\hat{E}}\left[  V(t+\delta,X_{t+\delta}^{t,x,v})+\int
_{t}^{t+\delta}f(s,X_{s}^{t,x,v},v_{s})ds+\int_{t}^{t+\delta}g_{ij}%
(s,X_{s}^{t,x,v},v_{s})d\langle B^{i},B^{j}\rangle_{s}\right]  ,
\]
which implies%
\[
V(t,x)\geq \underset{v\in \mathcal{U}^{t}[t,t+\delta]}{\inf}\mathbb{\hat{E}%
}\left[  V(t+\delta,X_{t+\delta}^{t,x,v})+\int_{t}^{t+\delta}f(s,X_{s}%
^{t,x,v},v_{s})ds+\int_{t}^{t+\delta}g_{ij}(s,X_{s}^{t,x,v},v_{s})d\langle
B^{i},B^{j}\rangle_{s}\right]  .
\]

Now we prove the inequality in the opposite direction. For each given
$v\in \mathcal{U}^{t}[t,t+\delta]$, we only need to prove%
\begin{equation}
V(t,x)\leq \mathbb{\hat{E}}\left[  V(t+\delta,X_{t+\delta}^{t,x,v})+\int
_{t}^{t+\delta}f(s,X_{s}^{t,x,v},v_{s})ds+\int_{t}^{t+\delta}g_{ij}%
(s,X_{s}^{t,x,v},v_{s})d\langle B^{i},B^{j}\rangle_{s}\right]  . \label{e3-19}%
\end{equation}
Let $\xi^{k}$ be defined in (\ref{e3-15}). By Theorem \ref{th15}, for each
$x_{j}^{k}$, there exists a $v^{j,k}\in \mathcal{U}^{t}[t+\delta,T]$ such that%
\begin{equation}
|V(t+\delta,x_{j}^{k})-J(t+\delta,x_{j}^{k},v^{j,k})|\leq k^{-1}.
\label{e3-20}%
\end{equation}
Set $\tilde{v}_{s}^{k}=\sum_{j=1}^{N_{k}}I_{A_{j}^{k}}v_{s}^{j,k}%
I_{[t+\delta,T]}(s)$, it is easy to verify that%
\[
J(t+\delta,\xi^{k},\tilde{v}^{k})=\sum_{j=1}^{N_{k}}I_{A_{j}^{k}}%
J(t+\delta,x_{j}^{k},v^{j,k})\text{ and }V(t+\delta,\xi^{k})=\sum_{j=1}%
^{N_{k}}I_{A_{j}^{k}}V(t+\delta,x_{j}^{k}).
\]
Thus we obtain%
\begin{equation}
|V(t+\delta,\xi^{k})-J(t+\delta,\xi^{k},\tilde{v}^{k})|\leq k^{-1}
\label{e3-21}%
\end{equation}
by (\ref{e3-20}). Set $v_{s}^{k}=v_{s}I_{[t,t+\delta)}(s)+\tilde{v}_{s}%
^{k}I_{[t+\delta,T]}(s)$. Then%
\begin{equation}
V(t,x)\leq \mathbb{\hat{E}}\left[  J(t+\delta,X_{t+\delta}^{t,x,v},\tilde
{v}^{k})+\int_{t}^{t+\delta}f(s,X_{s}^{t,x,v},v_{s})ds+\int_{t}^{t+\delta
}g_{ij}(s,X_{s}^{t,x,v},v_{s})d\langle B^{i},B^{j}\rangle_{s}\right]  .
\label{e3-22}%
\end{equation}
By (\ref{e3-15}), (\ref{e3-17}), (\ref{e3-18}) and (\ref{e3-21}), we obtain
(\ref{e3-19}) by taking $k\rightarrow \infty$ in (\ref{e3-22}). Thus we obtain
the dynamic programming principle.
\end{proof}

By using the dynamic programming principle, we study the properties of
$V(\cdot,\cdot)$ in $t$.

\begin{proposition}
\label{pr18}Let Assumptions (H1)-(H3) hold. Then, for each $t<T$, $\delta \leq
T-t$ and $x\in \mathbb{R}^{n}$, we have%
\[
|V(t,x)-V(t+\delta,x)|\leq C(1+|x|^{2})\sqrt{\delta},
\]
where $C>0$ depends on $T$, $\bar{\sigma}^{2}$ and $L$.
\end{proposition}

\begin{proof}
By Theorem \ref{th17}, we have%
\[
V(t,x)=\underset{v\in \mathcal{U}^{t}[t,t+\delta]}{\inf}\mathbb{\hat{E}}\left[
V(t+\delta,X_{t+\delta}^{t,x,v})+\int_{t}^{t+\delta}f(s,X_{s}^{t,x,v}%
,v_{s})ds+\int_{t}^{t+\delta}g_{ij}(s,X_{s}^{t,x,v},v_{s})d\langle B^{i}%
,B^{j}\rangle_{s}\right]  .
\]
By Proposition \ref{pr16}, Theorem \ref{th-ex} and H\"{o}lder's inequality, we
get%
\begin{align*}
&  \mathbb{\hat{E}}\left[  |V(t+\delta,X_{t+\delta}^{t,x,v})-V(t+\delta
,x)|\right] \\
&  \leq C(1+|x|+(\mathbb{\hat{E}}\left[  |X_{t+\delta}^{t,x,v}|^{2}\right]
)^{1/2})(\mathbb{\hat{E}}\left[  |X_{t+\delta}^{t,x,v}-x|^{2}\right]
)^{1/2}\\
&  \leq C(1+|x|^{2})\sqrt{\delta},
\end{align*}
where $C>0$ depends on $T$, $\bar{\sigma}^{2}$ and $L$. It follows from (H1)
and Theorem \ref{th-ex} that%
\begin{align*}
&  \mathbb{\hat{E}}\left[  \left \vert \int_{t}^{t+\delta}f(s,X_{s}%
^{t,x,v},v_{s})ds+\int_{t}^{t+\delta}g_{ij}(s,X_{s}^{t,x,v},v_{s})d\langle
B^{i},B^{j}\rangle_{s}\right \vert \right] \\
&  \leq C\int_{t}^{t+\delta}(1+\mathbb{\hat{E}}[|X_{s}^{t,x,v}|^{2}])ds\\
&  \leq C(1+|x|^{2})\delta,
\end{align*}
where $C>0$ depends on $T$, $\bar{\sigma}^{2}$ and $L$. Thus we obtain the
desired result.
\end{proof}

\section{HJB equation}

In this section, we show that the value function $V(\cdot,\cdot)$ satisfies
the following HJB equation:%
\begin{equation}
\left \{
\begin{array}
[c]{l}%
\partial_{t}V(t,x)+\underset{v\in U}{\inf}H(t,x,\partial_{x}V(t,x),\partial
_{xx}^{2}V(t,x),v)=0,\\
V(T,x)=\Phi(x),\text{ }x\in \mathbb{R}^{n},
\end{array}
\right.  \label{e4-1}%
\end{equation}
where $(t,x,p,A,v)\in \lbrack0,T]\times \mathbb{R}^{n}\times \mathbb{R}^{n}%
\times \mathbb{S}_{n}\times U$,%
\[%
\begin{array}
[c]{c}%
H(t,x,p,A,v)=G(F(t,x,p,A,v))+\langle p,b(t,x,v)\rangle+f(t,x,v),\\
F_{ij}(t,x,p,A,v)=(\sigma^{T}(t,x,v)A\sigma(t,x,v))_{ij}+2\langle
p,h_{ij}(t,x,v)\rangle+2g_{ij}(t,x,v).
\end{array}
\]

\begin{definition}
\label{de19}(\cite{CIP}) A function $V(\cdot,\cdot)\in C\left(  [0,T]\times
\mathbb{R}^{n}\right)  $ is called a viscosity subsolution (resp.
supersolution) to (\ref{e4-1}) if $V(T,x)\leq \Phi(x)$ (resp. $V(T,x)\geq
\Phi(x)$) for each $x\in \mathbb{R}^{n}$, and for each given $\left(
t,x\right)  \in \lbrack0,T)\times \mathbb{R}^{n}$, $\phi \in C_{Lip}%
^{1,2}([0,T]\times \mathbb{R}^{n})$ such that $\phi \left(  t,x\right)  =V(t,x)$
and $\phi \geq V$ (resp. $\phi \leq V$) on $[0,T]\times \mathbb{R}^{n}$, we have
\[
\partial_{t}\phi(t,x)+\underset{v\in U}{\inf}H(t,x,\partial_{x}\phi
(t,x),\partial_{xx}^{2}\phi(t,x),v)\geq0\text{ (resp. }\leq0\text{)}.
\]

A function $V(\cdot,\cdot)\in C\left(  [0,T]\times \mathbb{R}^{n}\right)  $ is
called a viscosity solution to (\ref{e4-1}) if it is both a viscosity
subsolution and a viscosity supersolution to (\ref{e4-1}).
\end{definition}

\begin{remark}
$C_{Lip}^{1,2}([0,T]\times \mathbb{R}^{n})$ denotes the set of real-valued
functions that are continuously differentiable up to the first order (resp.
second order) in $t$-variable (resp. $x$-variable) and whose derivatives are
Lipschitz functions.
\end{remark}

\begin{theorem}
\label{th20}Let Assumptions (H1)-(H3) hold. Then the value function
$V(\cdot,\cdot)$ is the unique viscosity solution to the HJB equation
(\ref{e4-1}).
\end{theorem}

\begin{proof}
By Propositions \ref{pr16} and \ref{pr18}, we have $V(\cdot,\cdot)\in C\left(
[0,T]\times \mathbb{R}^{n}\right)  $. Now, we prove that $V(\cdot,\cdot)$ is a
viscosity subsolution to (\ref{e4-1}).

For each fixed $\left(  t,x\right)  \in \lbrack0,T)\times \mathbb{R}^{n}$,
$\phi \in C_{Lip}^{1,2}([0,T]\times \mathbb{R}^{n})$ such that $\phi \left(
t,x\right)  =V(t,x)$ and $\phi \geq V$, by Theorem \ref{th17}, we deduce that,
for $\delta \leq T-t$,%
\[
\phi \left(  t,x\right)  \leq \underset{u\in \mathcal{U}^{t}[t,t+\delta]}{\inf
}\mathbb{\hat{E}}\left[  \phi(t+\delta,X_{t+\delta}^{t,x,u})+\int
_{t}^{t+\delta}f(s,X_{s}^{t,x,u},u_{s})ds+\int_{t}^{t+\delta}g_{ij}%
(s,X_{s}^{t,x,u},u_{s})d\langle B^{i},B^{j}\rangle_{s}\right]  .
\]
Applying It\^{o}'s formula to $\phi(s,X_{s}^{t,x,u})$ on $[t,t+\delta]$, we
get%
\begin{align*}
&  \mathbb{\hat{E}}\left[  \phi(t+\delta,X_{t+\delta}^{t,x,u})-\phi \left(
t,x\right)  +\int_{t}^{t+\delta}f(s,X_{s}^{t,x,u},u_{s})ds+\int_{t}^{t+\delta
}g_{ij}(s,X_{s}^{t,x,u},u_{s})d\langle B^{i},B^{j}\rangle_{s}\right] \\
&  =\mathbb{\hat{E}}\left[  \int_{t}^{t+\delta}\Lambda_{1}(s,X_{s}%
^{t,x,u},u_{s})ds+\int_{t}^{t+\delta}\Lambda_{2}^{ij}(s,X_{s}^{t,x,u}%
,u_{s})d\langle B^{i},B^{j}\rangle_{s}\right]  ,
\end{align*}
where%
\[%
\begin{array}
[c]{rl}%
\Lambda_{1}(s,x,v)= & \partial_{t}\phi(s,x)+\langle b(s,x,v),\partial_{x}%
\phi(s,x)\rangle+f(s,x,v),\\
\Lambda_{2}^{ij}(s,x,v)= & \frac{1}{2}F_{ij}(s,x,\partial_{x}\phi
(s,x),\partial_{xx}^{2}\phi(s,x),v).
\end{array}
\]
Thus we obtain%
\begin{equation}
\underset{u\in \mathcal{U}^{t}[t,t+\delta]}{\inf}\mathbb{\hat{E}}\left[
\int_{t}^{t+\delta}\Lambda_{1}(s,X_{s}^{t,x,u},u_{s})ds+\int_{t}^{t+\delta
}\Lambda_{2}^{ij}(s,X_{s}^{t,x,u},u_{s})d\langle B^{i},B^{j}\rangle
_{s}\right]  \geq0. \label{e4-2}%
\end{equation}

Since $\phi \in C_{Lip}^{1,2}([0,T]\times \mathbb{R}^{n})$, we have%
\[%
\begin{array}
[c]{l}%
|\partial_{t}\phi(s,X_{s}^{t,x,u})-\partial_{t}\phi(s,x)|+|\partial_{x}%
\phi(s,X_{s}^{t,x,u})-\partial_{x}\phi(s,x)|\\
+|\partial_{xx}^{2}\phi(s,X_{s}^{t,x,u})-\partial_{xx}^{2}\phi(s,x)|\leq
C|X_{s}^{t,x,u}-x|,
\end{array}
\]
where $C>0$ depends on $\phi$. Then, by (H1), we get%
\begin{align*}
&  |\Lambda_{1}(s,X_{s}^{t,x,u},u_{s})-\Lambda_{1}(s,x,u_{s})|+|\Lambda
_{2}^{ij}(s,X_{s}^{t,x,u},u_{s})-\Lambda_{2}^{ij}(s,x,u_{s})|\\
&  \leq C(1+|x|^{2}+|X_{s}^{t,x,u}|^{2})|X_{s}^{t,x,u}-x|,
\end{align*}
where $C>0$ depends on $L$ and $\phi$. By Theorem \ref{th-ex} and H\"{o}lder's
inequality, we obtain%
\begin{align*}
&  \mathbb{\hat{E}}\left[  \int_{t}^{t+\delta}(|\Lambda_{1}(s,X_{s}%
^{t,x,u},u_{s})-\Lambda_{1}(s,x,u_{s})|+|\Lambda_{2}^{ij}(s,X_{s}%
^{t,x,u},u_{s})-\Lambda_{2}^{ij}(s,x,u_{s})|)ds\right] \\
&  \leq C\int_{t}^{t+\delta}(1+|x|^{2}+(\mathbb{\hat{E}}[|X_{s}^{t,x,u}%
|^{4}])^{1/2})(\mathbb{\hat{E}}[|X_{s}^{t,x,u}-x|^{2}])^{1/2}ds\\
&  \leq C(1+|x|^{3})\delta^{3/2},
\end{align*}
where $C>0$ depends on $T$, $\bar{\sigma}^{2}$, $L$ and $\phi$. Thus we have%
\begin{equation}
\underset{u\in \mathcal{U}^{t}[t,t+\delta]}{\inf}\mathbb{\hat{E}}\left[
\int_{t}^{t+\delta}\Lambda_{1}(s,x,u_{s})ds+\int_{t}^{t+\delta}\Lambda
_{2}^{ij}(s,x,u_{s})d\langle B^{i},B^{j}\rangle_{s}\right]  \geq
-C(1+|x|^{3})\delta^{3/2}. \label{e4-3}%
\end{equation}

Set
\begin{equation}
\Lambda(s,x)=\inf_{v\in U}\{ \Lambda_{1}(s,x,v)+2G((\Lambda_{2}^{ij}%
(s,x,v))_{i,j=1}^{d})\}. \label{e4-4}%
\end{equation}
Then, by Proposition 4.1.4 in \cite{P2019}, we get%
\begin{align*}
&  \underset{u\in \mathcal{U}^{t}[t,t+\delta]}{\inf}\mathbb{\hat{E}}\left[
\int_{t}^{t+\delta}\Lambda_{1}(s,x,u_{s})ds+\int_{t}^{t+\delta}\Lambda
_{2}^{ij}(s,x,u_{s})d\langle B^{i},B^{j}\rangle_{s}\right] \\
&  \geq \int_{t}^{t+\delta}\Lambda(s,x)ds+\underset{u\in \mathcal{U}%
^{t}[t,t+\delta]}{\inf}\mathbb{\hat{E}}\left[  \int_{t}^{t+\delta}\Lambda
_{2}^{ij}(s,x,u_{s})d\langle B^{i},B^{j}\rangle_{s}-2\int_{t}^{t+\delta
}G((\Lambda_{2}^{ij}(s,x,u_{s}))_{i,j=1}^{d})ds\right] \\
&  =\int_{t}^{t+\delta}\Lambda(s,x)ds.
\end{align*}
By measurable selection theorem, there exists a deterministic control
$u^{\ast}\in \mathcal{U}^{t}[t,t+\delta]$ such that%
\[
\int_{t}^{t+\delta}\Lambda(s,x)ds=\int_{t}^{t+\delta}[\Lambda_{1}%
(s,x,u_{s}^{\ast})+2G((\Lambda_{2}^{ij}(s,x,u_{s}^{\ast}))_{i,j=1}^{d})]ds,
\]
which implies%
\begin{align*}
&  \underset{u\in \mathcal{U}^{t}[t,t+\delta]}{\inf}\mathbb{\hat{E}}\left[
\int_{t}^{t+\delta}\Lambda_{1}(s,x,u_{s})ds+\int_{t}^{t+\delta}\Lambda
_{2}^{ij}(s,x,u_{s})d\langle B^{i},B^{j}\rangle_{s}\right] \\
&  \leq \int_{t}^{t+\delta}\Lambda(s,x)ds+\mathbb{\hat{E}}\left[  \int
_{t}^{t+\delta}\Lambda_{2}^{ij}(s,x,u_{s}^{\ast})d\langle B^{i},B^{j}%
\rangle_{s}-2\int_{t}^{t+\delta}G((\Lambda_{2}^{ij}(s,x,u_{s}^{\ast}%
))_{i,j=1}^{d})ds\right] \\
&  =\int_{t}^{t+\delta}\Lambda(s,x)ds.
\end{align*}
Thus, by (\ref{e4-3}), we obtain%
\begin{equation}
\int_{t}^{t+\delta}\Lambda(s,x)ds\geq-C(1+|x|^{3})\delta^{3/2}, \label{e4-5}%
\end{equation}
where $C>0$ depends on $T$, $\bar{\sigma}^{2}$, $L$ and $\phi$. It is easy to
check that $\Lambda(s,x)$ defined in (\ref{e4-4}) is continuous in $s$. Thus,
by (\ref{e4-5}), we get%
\[
\Lambda(t,x)=\lim_{\delta \downarrow0}\frac{1}{\delta}\int_{t}^{t+\delta
}\Lambda(s,x)ds\geq0,
\]
which implies that $V(\cdot,\cdot)$ is a viscosity subsolution to
(\ref{e4-1}). Similarly, we can show that $V(\cdot,\cdot)$ is a viscosity
supersolution to (\ref{e4-1}). Thus $V(\cdot,\cdot)$ is a viscosity solution
to (\ref{e4-1}). The uniqueness is due to Theorem 3.5 in \cite{BBP} (see also
\cite{BL, HJX}).
\end{proof}

\  \   \

\end{document}